\documentclass[a4paper, twoside]{article}

\usepackage[
	top = 1.15 in, 
	bottom = 1.25 in,
	left = 1.15 in, 
	right = 1.15 in,
	includehead]{geometry}
\usepackage{scrextend}
\changefontsizes{11 pt}

\usepackage{microtype}
\usepackage[all]{nowidow}
\usepackage[utf8x]{inputenc}
\usepackage{microtype}
\usepackage{mathptmx}
\usepackage[T1]{fontenc}

\usepackage{booktabs}

\usepackage[utf8x]{inputenc}
\usepackage{amsfonts,amscd,amssymb,amsmath,amsrefs}
\usepackage{graphicx}
\usepackage[british]{babel}
\usepackage{caption}
\usepackage{mathdots}
\usepackage{mathtools} 
\usepackage{subfigure}

\usepackage{makeidx}
\usepackage{multicol}
\usepackage{array}
\usepackage[pdf,all]{xy}
\xyoption{line}
\CompileMatrices

\usepackage{sidecap}

\mathchardef\hy="2D 

\usepackage{pgf,tikz}
\usepackage{tikz-cd}
\usetikzlibrary{calc}
\usetikzlibrary{matrix,arrows}
\usepackage{stackrel}
\usepackage[shortlabels]{enumitem} 
\usepackage{stmaryrd} 
\usepackage{setspace} 
\spacing{1.15}
\usepackage{etoolbox}
\usetikzlibrary{trees}

\patchcmd{\section}{\normalfont}{\normalfont\large}{}{}

\usepackage[bb=ams, cal=euler, scr=rsfso , frak=euler]{mathalpha}

\usepackage{fancyhdr}
\fancypagestyle{references}
{\fancyhead[RE]{\small\it References}
\fancyhead[LO]{\small\it References}
\fancyhead[RO,LE]{\small\bf\thepage}
\fancyfoot[L,R,C]{}
}

\DeclareMathOperator{\cone}{cone}

\newcommand{\Ad}{\operatorname{Ad}}
 
\usepackage{enumitem}
\setlist[enumerate]{label= (\arabic*)}

\newenvironment{tenumerate}{
 \begin{enumerate}
  \setlength{\itemsep}{0pt}
  \setlength{\parskip}{0pt}
}{\end{enumerate}}

\newenvironment{titemize}{
\begin{itemize}
  \setlength{\itemsep}{0pt}
  \setlength{\parskip}{0pt}
}{\end{itemize}}

\definecolor{trinityblue}{rgb}{0.05, 0.45,0.75}
\definecolor{newcol}{rgb}{0,0,0}
\definecolor{deepblue}{rgb}{0.05, 0.45,0.75}
\DeclareTextFontCommand{\new}{\color{black}\em}

\usepackage[pdftex, colorlinks,bookmarks 
= true,bookmarksnumbered = true]{hyperref}

\hypersetup{colorlinks,linkcolor={deepblue},
	citecolor={deepblue},urlcolor={deepblue}}  
\usepackage{sectsty}
\chapterfont{\color{newcol}}  
\sectionfont{\color{newcol}}  
\subsectionfont{\color{newcol}}  

\makeatletter
\AtBeginDocument{%
  \let\[\@undefined
  
\DeclareRobustCommand{\[}{\begin{equation}}%
  \let\]\@undefined
  
\DeclareRobustCommand{\]}{\end{equation}}%
}
\makeatother 
\mathtoolsset{showonlyrefs,showmanualtags}

\usepackage{amsthm}
\usepackage{thmtools}
\newtheoremstyle{mytheorem}
  {\topsep}   
  {\topsep}   
  {\itshape}  
  {0pt}       
  {\bfseries\color{newcol}} 
  {\color{newcol}}         
  {5pt plus 1pt minus 1pt} 
  {}          
  
\theoremstyle{mytheorem}
\newtheorem{theorem}{Theorem}[section]
\newtheorem{proposition}[theorem]{Proposition}
\newtheorem{corollary}[theorem]{Corollary}
\newtheorem{lemma}[theorem]{Lemma}
\newtheorem*{conj*}{Conjecture}

 \newtheoremstyle{introthm}
  {\topsep}   
  {\topsep}   
  {\itshape}  
  {0pt}       
  {\bfseries\color{newcol}} 
  {\color{newcol}{.}}         
  {5pt plus 1pt minus 1pt} 
  {}          

\theoremstyle{introthm}

\newtheoremstyle{mydefinition}
  {\topsep}   
  {\topsep}   
  {}  
  {0pt}       
  {\bfseries\color{newcol}} 
  {\color{newcol}}         
  {5pt plus 1pt minus 1pt} 
  {}          
  
\theoremstyle{mydefinition}

\newtheoremstyle{mydefinition2}
  {\topsep}   
  {\topsep}   
  {}  
  {0pt}       
  {\bfseries\color{newcol}} 
  {\color{newcol}{.}}         
  {5pt plus 1pt minus 1pt} 
  {}          
  
\theoremstyle{mydefinition2}
\newtheorem*{definition*}{Definition}
\newtheorem*{remark*}{Remark}
\newtheorem*{obs*}{Observation}
\newtheorem*{example*}{Example}
\newtheorem{remark}{Remark}

\theoremstyle{plain} 
\newcommand{\thistheoremname}{}
\newtheorem{genericthm}[theorem]{\thistheoremname}

\newcommand{\imor}{\interleave\kern-.45em\longrightarrow}

\newcommand{\vv}{\vert}

\newcommand{\Der}{\operatorname{Der}}
\newcommand{\HH}{\mathrm{HH}}

\newcommand{\Com}{\mathsf{Com}}
\newcommand{\Ass}{\mathsf{Ass}}
\newcommand{\Lie}{\mathsf{Lie}}
\newcommand{\Gers}{\mathsf{Gers}}
\newcommand{\PCalc}{\mathsf{PreCalc}}
\newcommand{\Calc}{\mathsf{Calc}}
\newcommand{\D}{\mathsf{Diff}}
\newcommand{\Br}{\mathsf{Br}}

\newcommand{\Vect}{\mathsf{Vect}}

\newcommand{\SMod}{{}_\Sigma\mathsf{dgMod}}

\newcommand{\Alg}{\mathsf{Alg}}

\renewcommand{\tt}{\otimes}

\newcommand{\Tor}{\operatorname{Tor}}
\newcommand{\End}{\operatorname{End}}

\newcommand{\Ho}{\operatorname{Ho}}

\definecolor{sqsqsq}{rgb}{0.13,0.13,0.13}
\definecolor{aqaqaq}{rgb}{0.63,0.63,0.63}

\newcommand{\HC}{\mathrm{HC}}
\newcommand{\uk}{\underline{\kk}}

\newcommand{\Uupsilon}{\Upsilon}

\newcommand{\antishriek}{\text{\raisebox{\depth}{\textexclamdown}}}

\newcommand{\T}{\Theta}
\newcommand{\X}{\mathcal{X}}

\newcommand{\NN}{\mathbb N}

\newcommand{\kk}{\Bbbk}

\definecolor{newterm-color}{RGB}{0, 0, 0}

\theoremstyle{mytheorem}
\newtheorem*{theorem*}{Theorem}

\theoremstyle{plain} 
\renewenvironment{abstract}{%
\small\begin{center}
\begin{minipage}{.9\textwidth}
\small
}
{\par\noindent\end{minipage}\end{center}\vspace{3 em}}
\makeatletter
\renewcommand\@maketitle{%
\hfill
\begin{center}\begin{minipage}{0.9 	\textwidth}
\centering
\vskip 6em
\let\footnote\thanks 
{\LARGE \@title \par }
\vspace{1 em}
\vskip 1 em
{\large \@author \par}
\vspace{3.5 em}

\end{minipage}\end{center}
\par
}
\makeatother

\tikzcdset{arrow style=tikz, diagrams={>=stealth}}

\usepackage{titletoc}

\titlecontents{section}
[0.2em]
{\small}
{\thecontentslabel.\hspace{3pt}}
{}
{\enspace\titlerule*[0.5pc]{.}\contentspage}
\titlecontents*{subsection}
[1em]
{\footnotesize}
{\thecontentslabel. \hspace{3pt}}
{}
{}
[ --- \ ]
[]

\title{\vspace{- 1.25 in}\setstretch{0.85}{\textbf{The Tamarkin--Tsygan calculus of an
associative algebra \`a la Stasheff}}}
\author{\textsc{{Pedro Tamaroff}}}
\date{}
\begin{document}
\maketitle

\begin{abstract}
We show how to compute the Tamarkin--Tsygan
calculus of an associative algebra by providing,
for a given cofibrant replacement of it, a `small'
$\Calc_\infty$-model of its calculus, which we make
somewhat explicit at the level of $\Calc$-algebras.
To do this, we prove that the operad $\Calc$ is
inhomogeneous Koszul; to our best knowledge, this
result is new. We illustrate our technique by
carrying out some computations for two
monomial associative {algebras} 
using the cofibrant
replacement obtained by the author
in~\cite{Tamaroff}. 
\end{abstract}

\section{Introduction}
To every associative algebra we may associate its Hochschild homology 
and cohomology groups. These are a priori graded spaces, but in fact
are acted upon by several operads.
In the simplest level, Hochschild
cohomology is a graded commutative algebra under the cup product,
and in fact a Gerstenhaber algebra, and Hochschild homology is
a module for both of these algebra structures. It is the case that the
Hochschild complex admits higher brace operations~\cite{Braces},
refining the Gerstenhaber bracket, and the dg operad of such braces along with the cup product is quasi-isomorphic to the dg operad of singular chains on the little disks
operad. In this way McClure and Smith give, in~\cite{Braces}, a
solution to Deligne's conjecture. Another approach to the solution of the
conjecture, preceding that of McClure
and Smith, was
proposed by Tamarkin in~\cite{DeligneTam}. In~\cite{HinichTam}, Hinich provided further
details to the approach of Tamarkin.

As originally observed in~\cite{Daletskii}, there is another operad that acts on Hochschild cohomology and
homology, the 2-colored operad of Tamarkin--Tsygan
calculi~\cite{TT}*{\S 3.6},
which yields for every associative
algebra $A$ a pair
\[ \Calc_A = (\HH^*(A),\HH_*(A)),\]
which we call the \emph{Tamarkin--Tsygan calculus of $A$}. In this paper we focus our attention on giving formulas for
the action of it in terms of cofibrant resolutions in $\Alg$, the
category of dga algebras with the projective model structure or, what
is the same, a \emph{homotopy invariant} description of the action
on the chain level. Initially, we focused our attention on the
Gerstenhaber bracket on Hochschild cohomology, originally defined
in~\cite{Gers}, since computing the resulting Gerstenhaber algebra
structure on Hochschild cohomology has been of interest~\cite{MSA,Vol,NegWit,NegWitBr,WitBr}, and is agreed to be a
non-trivial task; the reason for this is nicely explained in~\cite{MSA}.

Quillen~\cite{QuillenCyclic}*{Part II, \S 3} identified
the Hochschild cochain complex of an algebra with the space
of coderivations of its bar construction. Later, in the
same spirit, Stasheff~\cite{Sheff} gave a definition of the
Gerstenhaber  bracket of an associative algebra as the Lie bracket
of coderivations of its bar construction, which deserves to be
thought of as intrinsic to the category of associative
algebras. Interest for a description of the Tamarkin--Tsygan
calculus of an associative algebra \`a la Stasheff appears
in~\cite{ZimmerRem}*{Remark 7}. Our main result here
is giving such an intrinsic description of the action of
$\mathsf{Calc}$ that also lends itself to computations.

To do this, consider a quasi-free associative algebra $B=(TV,d)$,
and the following complexes of nc poly vector fields
and of nc differential forms on $B$, respectively,
\[\X^*(B) =
	\cone(\Ad {}{\colon} B\longrightarrow \Der(B)),\quad
  \T_*(B) =
	\cone(C{}{\colon}V\otimes B\longrightarrow B){}{,}
		\]
where $C$ is the commutator map. The space $\T_*(B)$ is spanned by
nc-forms $\omega =b+b'dv$ where $b,b'\in B$ and $v\in V$ and
$\X^*(B)$ is spanned by nc-fields $X=\lambda + f$ where $\lambda
\in B$ and $f$ is a derivation.

\begin{theorem}
The Tamarkin--Tsygan calculus $\Calc_A$ of an associative algebra
$A$ can be computed using the datum
\[ (\X^*(B),\T_*(B)) \]
of nc poly vector fields and nc differential forms obtained from
a model $ B \longrightarrow A$. In terms these two complexes, if we write a poly vector field by $X=\lambda+f$ and a differential form
by $\omega = b+b'dv$, then

\begin{enumerate}
\setlength{\itemsep}{0pt}
  \setlength{\parskip}{0pt}
 \item The cup product can be computed through a brace operation
 		$\{X,Y;d\}$.
 \item The Lie bracket is the usual bracket of derivations{}{.}
 \item The boundary of $\omega$ is computed as
		\[d\omega = \sum_{i=1}^n (-1)^\varepsilon
				v_{i+1}\cdots v_nv_1\cdots v_{i-1}dv_i. \]
\end{enumerate}
\end{theorem}

\noindent To achieve this, we prove the seemingly new result
that the colored operad $\Calc$ is inhomogeneous Koszul. The
theory developed in~\cite{BV} to obtain resolutions of
inhomogeneous Koszul operads then allows us to produce a
cofibrant model $\Calc_\infty$ for $\Calc$ and endow the pair
\[(\X^*(B),\T_*(B))\] with a $\Calc_\infty$-algebra
structure.

\begin{theorem}
The operad $\Calc$ is inhomogeneous Koszul.
It admits a cofibrant model $\,\Calc_\infty$
of $\Calc$ with underlying symmetric sequence
isomorphic to $T(\delta)\circ\PCalc^{\emph{\antishriek}} $.
\end{theorem}

Here $\PCalc^\antishriek$ is the Koszul dual cooperad of the Koszul
operad $\PCalc$ and $T(\delta)$ is a polynomial algebra with
generator $\delta$ of degree $2$. Having this model for homotopy
coherent calculi, we gather the following results from~\cite{KS,TTD}:

\begin{theorem}
There is a dg colored operad $\mathsf{KS}$, the
Kontsevich--Soibelman operad, and a topological
colored operad $\mathsf{Cyl}$ such that
\begin{tenumerate}
\item The operad $C_*(\mathsf{Cyl})$ is formal and its homology
		is $\Calc$.
\item There is a quasi-isomorphism $\mathsf{KS} \longrightarrow
 		C_*(\mathsf{Cyl})$ of dg-operads.
\item $\mathsf{KS}$ acts on the pair $(C^*(-),C_*(-))$ in such a
		way that{}{.}
\item On homology we obtain the usual $\Calc$-algebra structure on
		$(\HH^*(-),\HH_*(-))$.
\end{tenumerate}
In particular, for every cofibrant replacement $Q$ of $\Calc$, the
pair $(C^*(A),C_*(A))$ is a $Q$-algebra which on homology gives
$\Calc_A$.\qed
\end{theorem}

Specializing this to $Q=\Calc_\infty$, we obtain on the pair
$(C^*(A),C_*(A))$ a $\Calc_\infty$-algebra which on homology gives
$\Calc_A$ and which we write $\Calc_{\infty,A}$. With this at hand,
the homotopy theory developed in~\cite{BV} for inhomogeneous Koszul
operads implies the following theorem.

\begin{theorem}
For any quasi-free model $B$ of $A$, the pair $(\X^*(B),\T_*(B))$
admits a $\Calc_\infty$-algebra structure which is
$\infty$-quasi-isomorphic to the $\Calc_\infty$-algebra
$\Calc_{\infty,A}$ so that the action of degree $0$ elements in
$\Calc_\infty$ is given by the formulas above. In particular,
this structure recovers $\Calc_A$ by taking homology.
\end{theorem}

We remark that computing  cofibrant resolutions in $\Alg$ is
remarkably complicated. The author has solved the problem of
computing models of \emph{monomial} algebras in~\cite{Tamaroff} and,
as explained there, using this and some ideas of deformation theory,
it is possible to attempt to compute models of certain algebras
with a chosen Gr\"obner basis, although the general description
of the minimal model of such algebras is, at the moment, missing.


In Section~\ref{sec:1} we recall the elements of Hochschild
(co)homology and necessary background on models of associative
algebras. In Section~\ref{sec:2} where we recall the basics of
(colored) operads, their algebras, and the relevant Koszul duality
theory for them, which we use to give a inhomogeneous Koszul model
of $\Calc$. In Section~\ref{sec:3}, we introduce the main algebraic
structure of our paper, Tamarkin--Tsygan calculi. We then give two
complexes to compute Hochschild (co)homology in terms of models,
and give the promised formulas for the Tamarkin{}{--}Tsygan calculus
of an algebra in terms of these. Finally, in Section~\ref{sec:4},
following the notion of Koszul $A_\infty$-algebras of
Berglund--B\"{o}rjeson~\cite{Alexander}, we show that the
Tamarkin--Tsygan calculus of an $A_\infty$-Koszul algebra
is dual to that of its Koszul dual algebra, extending a
result of Herscovich~\cite{Hers}. In Section~\ref{sec:4}
we give some examples of computations to illustrate our methods.

In what follows $\kk$ is a field of characteristic zero,
and all unadorned $\tt$ and $\hom$ are with respect to this
base field. All algebras are non-unital and non-negatively
homologically graded unless stated otherwise, and are defined
over the base field. Implicit signs follow the Koszul sign rule,
and we make them as explicit as we can. We will adhere to the
convention that $A$ will always denote a non-dg associative
algebra concentrated in homological degree $0$, while $B$ will
always denote a dg associative algebra concentrated in non-negative
degrees. We write $V^\#$ for the graded dual of a graded vector space.

\subsection*{Acknowledgements}
I am pleased to thank Vladimir
Dotsenko for his constant support during the preparation of
these notes, in particular for useful remarks and suggestions.
We thank R.~Campos, M.~Linckelmann, J.~Bellier-Mill\`{e}s
and G.~Zhou for useful comments and discussions,
and organizers of the meeting \emph{Higher structures} (2019, CIRM at Luminy) where many useful
conversations that contributed to the improvement of this paper took place.
I am indebted to Bernhard Keller for his
careful reading of a first version of this
paper, and for his very useful comments
and corrections during a visit of the author
to Paris~VII\@.
The second version of these notes was
produced while I was visiting
Alexander Berglund at the University of
Stockholm. I thank Alexander for very useful
conversations and remarks that allowed me to
improve the article, and the University
of Stockholm for providing excellent working
conditions. Finally, I thank
several anonymous referees for their
comments, questions, suggestions and corrections
that significantly improved this paper.

\section{Preliminaries}\label{sec:1}

In this section we quickly recall the classical
definitions of Hochschild homology and cohomology,
and cyclic homology of algebras.
For details, the reader is referred to
\cite{Loday} and \cite{McCleary}. We also recall the basics on models of associative algebras; these are
the objects we use to compute
their calculus. The reader familiar
with the theory may wish to skip
this section.  Throughout, we fix
an associative algebra $A$, which we always assume is concentrated in homological degree zero.

\subsection{Hochschild (co)homology}

Let us begin by
recalling what the Hochschild (co)chain
complex of a dga algebra
$(B,\partial_B)$ looks like.
For details, we
refer the reader to the book~\cite{McCleary}*{Section 7.1}
and also to~\cite{Hoss}*{Sections 1 and 2}.
To this end form the first quadrant double complex $C_*(B)_*$
so that for each $p,q\in\mathbb N_0$ we
have
\[
C_p(B)_q = (B\tt (s\overline{B})^{\tt p})_{p+q}, \]
the space of elements of $B\tt(s\overline{B})^{\tt p}$
of total homological degree $p+q$. Observe that
since we have shifted the degree of $\overline{B}$,
a generic element of bidegree $(p,q)$ is of the
form $b_0[b_1\vv\cdots \vv b_p]$ where
$\sum_{i=0}^p |b_i|= q$.

The horizontal differential is the Hochschild differential that
incorporates the Koszul signs of $B$, so that for
a generic basis element $b_0[b_1\vv\cdots\vv b_p]$ of
our double complex we have that

\begin{align*} d_h(b_0[b_1\vv\cdots\vv b_p]) =&
   -\overline{b_0}b_1[b_2\vv\cdots \vv b_p] 
   +\sum_{i=1}^{p-1} \overline{b_0}[\overline{b_1}\vv
   \cdots \vv \overline{b_i}b_{i+1}\vv\cdots \vv b_p]\\
   &+
     (-1)^{a|b_p|+|b_0|}b_pb_0[\overline{b_1}\vv\cdots \vv \overline{b_{p-1}}]{}{,}
\end{align*}

where 
we use the usual notation $\overline{b} =
(-1)^{|b|+1} b$, and $a = |b_0|+|b_1|+\cdots+|b_{p-1}|+p-1$. 
The vertical differential is simply
the induced differential on the tensor product
of complexes $B\otimes (s\overline{B})^{\tt p}$,
which incorporates Koszul signs and suspension signs.
Concretely, for
a generic basis element $b_0[b_1\vv\cdots\vv b_p]$ we have that
\[ d_v(b_0[b_1\vv\cdots\vv b_p]) =
   db_0[b_1\vv \cdots \vv b_p]
   -\sum_{i=1}^{p-1} \overline{b_0}[\overline{b_1}\vv
   \cdots \vv \overline{b_i}\vv db_{i+1}\vv\cdots \vv b_p]{}{.} \]

Dually, we can form the double complex
$C^*(B)^*$ so that for each $p,q\in \NN_0$ we have

\[ C^p(B)^q = 
	\hom((s\overline{B})^{\tt p},B)^{p+q},\]	
where the outer index indicates the cohomological
degree of a map. In this way, elements in $C^p(B)^q$
correspond to maps $\overline{B}^{\tt p} \longrightarrow B$
that lower the homological degree by $q$.
Observe this complex lives
in the first two quadrants, since $q$ can be negative.
For a homogeneous map 
$f {}{\colon} (s\overline{B})^{\tt p} \longrightarrow B$ and a basis
element $[b_1\vv \cdots \vv b_p]$ of $(s\overline{B})^{\tt p}$, we have

\begin{align*} 
(d_hf)[b_1\vv\cdots\vv b_p] =\ &
   (-1)^{|b_1||f|} {b_1}f[b_1\vv\cdots \vv b_p] 
+(-1)^{|f|}\sum_{i=1}^{p-1} f[\overline{b_1}\vv
   \cdots \vv \overline{b_i}b_{i+1}\vv\cdots \vv b_p]\\
   &+
     (-1)^{|f|}f[\overline{b_1}\vv\cdots \vv \overline{b_{p-2}}]b_{p-1}.
\end{align*}

\subsection{Models of associative algebras}

Let us write $\Alg$ for the category of
dga algebras. A surjective
quasi-isomorphism $B\longrightarrow A$ from
a quasi-free dg algebra $B=(TV,\partial_B)$ is a
\emph{model of $A$}. One usually calls $B$ a model of $A$,
without explicit mention to the map $B\longrightarrow A$
which is understood from context. We say a model is \emph{minimal} if $B$ is
\begin{tenumerate}
\item \emph{triangulated}: there is a gradation $V =
\bigoplus_{j\geqslant 1} V(j)$ such that for each $j\in\NN$,
we have that $ \partial_B(V^{(j+1)}) \subseteq T(V^{(\leqslant j)})$, and
\item its differential is \emph{decomposable}, that is, $\partial_B(V) \subseteq (TV)^{\geqslant 2}$.
\end{tenumerate}
There is a model structure on $\Alg$ whose weak equivalences
are quasi-isomorphisms, the fibrations are the degree-wise
surjections, and the cofibrant algebras are the retracts of
triangulated quasi-free algebras; see~\cite{HTHA}. In particular,
minimal algebras are cofibrant, and may be used to cofibrantly
resolve objects in $\Alg$. It is important
to recall that not every dga algebra admits
a minimal model; see the unpublished notes~\cite{KellerMod} for
details.
\[
\begin{tikzcd}[row sep = 4pt]
 B \arrow[dashed,dd]\arrow[dr]& {}  \\
 {} & A \\
 B' \arrow[ur]& {}
\end{tikzcd}\]
Minimal models of algebras, when they exist, are
unique up to unique isomorphism, meaning that any solid diagram
where the diagonal
arrows are minimal models, can be uniquely
completed to a homotopy commutative triangle where the vertical
dashed map is an isomorphism in $\Alg$. For brevity, we will use the term \emph{model} to speak about triangulated quasi-free algebras with homology concentrated in degree zero.

In~\cite{Tamaroff}, we obtain a description
of the minimal model of any monomial quiver algebra.
Concretely, the model is free over the quiver with
arrows the chains, also known as overlappings or ambiguities of the algebra, and the differential is given by
deconcatenation. If $\gamma$ is an
Anick chain of length $r$,
a decomposition of it is a sequence $(\gamma_1,\ldots,\gamma_n)$ of Anick
chains of lengths $r_1,\ldots,r_n$ so that
their concatenation, in this order, is $\gamma$, and
$r-1 = r_1+\cdots+r_n$. What follows is the main result
of that paper.

\begin{theorem}
For each monomial algebra $A$ there is a minimal model
$B\longrightarrow A$ where $B$ is the $\infty$-cobar construction on $\Tor_A(\kk,\kk)$. The differential $d$ is such that for a chain $\gamma\in\Tor_A(\kk,\kk)$,
\begin{align*} d\gamma = -\sum_{n\geqslant 2}
	(-1)^{\binom{n+1}{2}+|\gamma_1|}
	 	\gamma_1\gamma_2\cdots\gamma_n,
	 		\end{align*}
where the sum ranges through all possible decompositions of $\gamma$.\qed
	 \end{theorem}

\noindent Observe that the differential is manifestly decomposable, and
that the gradation on $\Tor_A$ provides us with a triangulation of $B$, so
that indeed $B$ is minimal. We give examples
of computation of this model in Section~\ref{sec:compute}
which we then use to compute the Tamarkin--Tsygan
calculus of $A$ explicitly.

\section{Elements of operad theory}\label{sec:2}
We now briefly recall here the
elements from the theory of colored operads,
their algebras, and Koszul duality theory which
allows us to give explicit models of algebras
`up to homotopy' of certain well behaved
operads. With this at hand, we can explain
in Section~\ref{sec:ncthings} how to
produce a homotopy coherent calculus
on nc forms and nc fields obtained from any
cofibrant replacement of an algebra.
We also prove that
the operad $\Calc$ controlling Tamarkin--Tsygan
calculi is inhomogeneous Koszul; this result
is probably known to experts, but we could
not find a proof in the literature.

\subsection{Symmetric operads}

Let us write $\SMod$ for the category
of $\Sigma$-modules, that is, collections
of graded $\kk$-modules
$X = \{X(n)\}_{n\geqslant 0}$
where for each $n\in\NN_0$ the space $X(n)$ is
a graded $S_n$-module. There is a monoidal product on
$\SMod$ that corresponds to the composition of
formal series, so that for two $\Sigma$-modules
$X$ and $Y$ with $Y(\varnothing)=0$
we have, for each finite set $I$
\[ (X\circ Y)(n) = \bigoplus_{k\geqslant 0} X(k)\otimes_{S_k} Y[k,n]{}{,} \]
where for each $n\in\NN$ and each $k\leqslant n$,
\[ Y[k,n]:=\bigoplus_{\lambda_1+\cdots+\lambda_k=n} \,{}_{S_\lambda}^{S_n}\uparrow Y(\lambda_1)\otimes\cdots\otimes Y(\lambda_k)\]
is the sum of induced representations from the subgroup
\[ S_\lambda = S_{\lambda_1}\times \cdots\times S_{\lambda_k}\subseteq S_n.\] Note it is a both an $S_n$-module
and an $S_k$-module (by permutation of the $k$ factors)
in a compatible way.

We like to think
of an element in $(X\circ Y)(n)$ as an $n$-corolla
---a rooted tree with a single vertex of degree
$n+1$--- with its root decorated with an element
of $X(n)$ and its leaves decorated with elements
of $Y$. This makes $\SMod$ into a monoidal category
with unit the symmetric sequence $(0,\kk,0,\ldots)$. Considering the particular case when
$Y=Y(1)$, we obtain an endofunctor $X{}{\colon}\Vect\longrightarrow \Vect$ such that
\[ X(V) = \bigoplus_{k\geqslant 0} X(k)\otimes_{S_k} V^{\otimes k}. \]

A symmetric operad is by definition an
 associative algebra in the monoidal category
$(\SMod, \circ,1)$. It is customary to
write $\gamma {}{\colon} P\circ P\longrightarrow P$ for
the product of $P$ which is instead called the
composition map of $P$. In this way, to every
element $(\mu;\nu_1,\ldots,\nu_t)\in P\circ P$
we assign an output $\gamma(\mu;\nu_1,\ldots,\nu_t)$ which we think of as grafting $\nu_1,\ldots,\nu_t$ at the leaves of $\mu$ to obtain a new operation
in $P$. In particular, one can consider the case
when all but one of the operations grafted at the
leaves are the unit $1\in P(1)$. This gives
us for each partial composition operations
\[ \circ_i {}{\colon} P(n) \otimes P(m) \longrightarrow P(m+n-1) \]
which are sufficient to describe $\gamma$: any
element $\gamma(a;b_1,\ldots,b_t)$  can be
obtained by iterated partial compositions. In this
language, associativity of $\gamma$ is equivalent
to the parallel and sequential composition axioms:
for operations $a,b$ and $c\in P$ with of respective
arities $l,m$ and $n$, we have that
\[\renewcommand{\arraystretch}{1.2}\begin{array}{c@{\quad}l} 
(a\circ_i b)\circ_{i+j-1} c = a\circ_i (b\circ_j c) & \text{for $i\in [l]$ and $j\in [m]$,} \\
(a\circ_i b)\circ_{k+l-1} c =(-1)^{|c||b|} (a\circ_k c)\circ_i b & \text{for $i<k\in [l]$.}
\end{array}\]
We will often make use of the partial composition
description of an operad to define the operads
of interest to us.

To every operad $P$ we can assign its
category of algebras. These are vector spaces $V$
endowed with a linear map $\gamma_V {}{\colon}P(V) \longrightarrow V$ that is compatible with the
composition law of $P$ in the sense that for
$v_1\otimes\cdots\otimes v_t\in V^{\otimes n_1}\otimes
\cdots\otimes V^{\otimes n_t}$
\[ \gamma_V(\gamma(a,b_1,\ldots,b_t),v_1,\ldots,v_t) = \gamma_V(a,\gamma_V(b_1,v_1),\ldots,
\gamma_V(b_t,v_t)), \]
and such that for each $\sigma\in S_n$ we have that \[ \gamma_V(a\cdot \sigma,v_1,\ldots,v_n) =
\gamma_V(a,v_{\sigma 1},\ldots,v_{\sigma n}).\]
We like to think of $\gamma_V(a;v_1,\ldots,v_n)$
as the result of applying operation $a$ to the
string $v_1\otimes\cdots \otimes v_n\in V^{\otimes n}$.
We also require that the unit of $P$ (if there is any)
acts as the identity endomorphism of the space $V$.

Alternatively, one can consider the endomorphism
operad $\End_V$ of a vector space~$V$ so that
for each $n\in\NN$ we have
\[ \End_V(n) = \End(V^{\otimes n},V) \]
and the operadic composition is given by
\[\gamma(f;g_1,\ldots,g_n) = f\circ (g_1\otimes \cdots g_t)\] while the symmetric group acts by permuting the variables of $V^{\otimes n}$. With this at hand,
a $P$-algebra structure on $V$ is given by a map
of operads $P\longrightarrow \End_V$.

A useful observation, that
allows us to define operads through generators
and relation, is
that operads can be defined as algebras over
a monad
\[ \mathcal T{}{\colon} \SMod\longrightarrow \SMod,\] which
we call the monad of trees. Concretely, for
each $\Sigma$-module $X$, we define a new
$\Sigma$-module $\mathcal T(X)$ so that for each $n\in\NN$
we have that $\mathcal T(X)(n)$ is spanned by isomorphism classes
of trees with $n$ leaves with each vertex $v$ decorated by en element in $X(k)$, where $k$ is the number of inputs of $v$. The multiplication map
\[ \mu{}{\colon} \mathcal T\circ \mathcal T
\longrightarrow \mathcal T \]
is obtained by erasing the boundaries of vertices in a `tree of trees'.

In this way, an operad is just an associative
algebra over $\mathcal T$ and, in particular, $\mathcal T(X)$ is
the free operad on $X$: $\mathcal T$ is the left adjoint
to the forgetful functor from the category
of operads to the category of $\SMod$. In this
way, it makes sense to say that an operad $P$ is generated by $X$: this means there is a surjective
map of operads
\[\mathcal T(X) \longrightarrow P\,.\] One then
defines the notion of a double sided ideal, as usual, and hence it makes sense to say that an
operad is generated by $X$ subject to some
$\Sigma$-module of relations~$R$, as in the
case of associative algebras.

Since it is useful for our purposes,
let us consider the example where $X= X(2)$
consists of the trivial representation of the
symmetric group $S_2$, with a generator $m$,
so that $\mathcal T(X)(2)$ consists of a single binary
tree which we identify with a commutative operation
$x_1x_2$. The relation $m\circ_1 m = m \circ_2 m$
which we can write more familiarly
as the following equation: \[ x_1(x_2x_3) = (x_1x_2)x_3.\]
It recovers the usual associativity of commutative algebras. We call $T(X)/(R)$ the
operad of commutative algebras and write it $\Com$.
Note that for each $n\in\NN$ the space $\Com(n)$
is the one dimensional trivial representation of
$S_n$, and it is spanned by an operation which
we can unambiguously write $x_1\ldots x_n$.
Similarly, we can consider the case $X(2)$ is the
regular representation of $S_2$, in which case
following the same prescription we would obtain
the operad $\Ass$ governing associative algebras.
In this case, for each $n\in\NN$ the space
$\Ass(n)$ is $n!$-dimensional and it is spanned
by the operation $x_1\ldots x_n$ along with
all possible permutations of its variables.

We
can similarly define the operad $\Lie$ governing
Lie algebras: it is generated by a single binary
operation $b$ spanning the sign representation of
$S_2$, subject to the Jacobi relation
$ (1+\tau+\tau^2)(b\circ_1 b)=0$
where $\tau = (123)\in S_3$. We can again write
this more familiarly as
\[ [[x_1,x_2],x_3]+ [[x_3,x_1],x_2] + [[x_2,x_3],x_1] = 0. \]
For each $n\in \NN$, the space $\Lie(n)$ has a
basis consisting of $(n-1)!$ Lie words
\[ [x_{\sigma 1},x_{\sigma 2},\ldots,x_{\sigma n}]{}{,}\] where $\sigma\in S_n$ fixes $1$, and where we are using the convention that we bracket a Lie word to the left. Thus, for example,
$[x_1,x_2,x_3] = [[x_1,x_2],x_3]$.

We can use these two operads to define
the operad $\Gers$ of Gerstenhaber
algebras. It generated by an associative
commutative product $m=x_1x_2$ in degree $0$
and an Jacobi bracket in degree $1$
$b=[x_1,x_2]$ subject to the Leibniz rule
$
 b\circ_1 m - m\circ_2 b - \sigma (m\circ_2 b)= 0,
$
where $\sigma = (12)\in S_3$, which we can again write in the familiar form
\[ [x_1x_2,x_3] = x_1[x_2,x_3] +x_2[x_1,x_3].\]

We remark that both in the Jacobi relation and
in the Leibniz rule are written without signs
since they are relations in the operad $\Lie$
and $\Gers$ respectively and, since $x_1$ and $x_2$
are merely placeholders, they carry no degree. When these are evaluated
at elements of a fixed algebra $V$, the Koszul
sign rule will create the appropriate signs.
For example, in our last equation, $x_1$ and $x_2$
are permuted in the last summand, so that when
evaluating this on some $v_1\otimes v_2\otimes v_3$
in $V^{\otimes 3}$, we will create a sign in the step where we permute $v_1$ and $v_2$ and then
apply $m\circ_2 b$:
\[ v_1\otimes v_2\otimes v_3 \longmapsto
 (-1)^{|v_1||v_2|} v_2\otimes v_1\otimes v_3.\]

Given an algebra $V$ over an operad
$P$, let us consider the subspace $P(V,M) \subseteq P(V\oplus M)$ of elements of the form
\[ (\mu;v_1,\ldots,v_{i-1},m,v_{i+1},\ldots,v_t),\quad\text{ where $\mu\in P$, $v_j\in V$ and $m\in M$}. \]
It is useful to think of this as a corolla with
its root decorated by an element of $P$, and
its leaves decorated by elements of $V$ except
one, which is decorated by an element of $M$,
which results on in an operation
\[ \mu {}{\colon} V^{\otimes (i-1)}\otimes M\otimes V^{\otimes (t-i)}
\longrightarrow M. \]
Note that $(P\circ P)(V,M)$ is naturally isomorphic
to $P(P(V),P(V,M))$. A left $V$-module is a vector space $M$
along with a structure map
\[ \gamma_M {}{\colon} P(V,M) \longrightarrow M \]
compatible with the algebra structure of $V$,
in the sense that, using the previous identification
 $\gamma_M\circ P(\gamma_V,\gamma_M)= \gamma_M\circ\gamma$.

 To get acquainted with this notion of
 operadic module, it is useful to note that
 for $V$ an associative algebra, a $V$-module
 is the same as a $V$-bimodule in the usual sense,
 the left and right actions given by the operations
 \[ mx = \gamma_M(\mu;m,x), \quad
 xm = \gamma_M(\mu;x,m). \]
 In case $V$ is a commutative algebra, a $V$-module is the
 same as a symmetric $V$-bimodule since in this
 case the symmetry of $\mu$ gives $xm=mx$, and for a Lie
 algebra~$V$, we recover the usual
 definition of a
 representation of a Lie algebra.

\section{Colored operads and Koszul duality} 

Operads allow us to 
describe algebras of certain kind. We can
embellish them to describe pairs of algebras
of a certain kind along with a module over
this algebra, as follows. 

Let $P$ be an operad, which we
think of as having color $\circ$ (white). We define
a $2$-colored operad $P^*$ with colors $\circ$
and $\bullet$ (black) as follows, where we use the
notation $P^*(\varepsilon;a,b)$ for the 
spaces of operations with $a$ inputs of color
$\circ$, $b$ inputs of color $\bullet$, and the single
output of color $\varepsilon\in \{\bullet,\circ\}$.
 
\begin{enumerate}
\item $P^*(\circ;a,b)= 
	\begin{cases}
		0 & \text{if $a\geqslant 1$,} \\
		P(b) & \text{if $a=0$.}
		\end{cases}$
\item $P^*(\bullet;a,b) = \begin{cases}
		0 & \text{if $a=0$ or $a\geqslant 2$,} \\
		(\partial P)(b+1) & \text{if $a=1$.}
		\end{cases}$
\end{enumerate}

Here $\partial$ is the pointing operator on
symmetric sequences: $\partial X(n)$ is a copy of $X(n+1)$ where $n+1$ is `marked' and $S_n$ acts on the remaining ones. Alternatively, we are taking the restriction of $X(n+1)$ for
the inclusion $S_n\longrightarrow S_{n+1}$ as the stabilizer of $n+1\in [n+1]$. We then
color the marked point black, as well as the output. In terms of finite sets,
we have that $(\partial X)(I) = 
X(I\sqcup \{I\})$ with the obvious
action of the symmetric group $S_I$. 

{{}} The composition law in $P^*$ is
dictated by the simple rule we may only graft a 
tree with root colored $\varepsilon$ into a 
leaf of the same color. Hence, we only have
compositions coming from those of $P$, and two
new compositions that either graft an operation
of $P$ into a white leaf of a 
marked operation of $\partial P$, or a new
composition that grafts a black root into a 
black leaf of $\partial P$, both resulting in
an operation in $P^*(\bullet;1,a)$.
 
{{}} If $V=(V_\circ,V_\bullet)$ is a pair of vector
spaces, which we think of put in color $(\circ,\bullet)$, then we can consider the colored
operad $\End_{V}$,
where for each $s,t\in\NN$ and $\varepsilon\in \{\circ,\bullet\}$,
\[ \End_{V}(\varepsilon;s,t) = \hom(V_\circ^{\otimes s}\otimes V_\bullet^{\otimes t}, V_\varepsilon).\] 
In this way, a $P^*$-algebra structure on $V$ is 
equivalent to a morphism of colored operads
\[ P^* \longrightarrow  \End_{V}. \]
With this at hand, we have the following elementary
result. We point the reader to~\cite{Anton} for
details. 

\begin{proposition}
A $P^*$-algebra structure on $V=(V_\circ,V_\bullet)$
is equivalent to the datum of
a $P$-algebra structure on $V_\circ$ and a $V_\circ$-module structure on 
$V_\bullet$.  \qed
\end{proposition}

Let us record the following result, which we will use later.
\begin{proposition}
The operad $P^*$ is Koszul if $P$ is Koszul.
\end{proposition}
\begin{proof}
This follows from the fact
the homology of the free $P^*$-algebra on
$(V_\circ,V_\bullet)$
is nothing
but the Koszul homology of the free $P$-algebra over
$V_\circ$ in the
first color and the Koszul homology of the free $P$-algebra over $V_\circ$ with coefficients in
the free $V_\circ$-module on $V_\bullet$, and
both these complexes are acyclic since $P$
is Koszul.  Since an operad is Koszul if and only if
its free algebras have trivial
Koszul (co)homology in non-negative degrees,
this proves the theorem.
\end{proof}

Observe that this implies both operads $\Com^*$ and
$\mathcal S^{-1}\Lie^*$ are Koszul. 
From the operad $\Gers$ we form the $2$-colored operad 
$\PCalc=(\Gers,\mathsf{M})$ by adding two operations $[v,x]$ 
and $v\cdot x$ in 
$\PCalc(\bullet;1,1)$, of degrees $0$ and $-1$ respectively, and relations as follows: 

\begin{tenumerate}
\item the operations $x_1x_2$ and $v\cdot x_1$ 
satisfy the same relations defining $\Com^*$,
\item the operations $[x_1,x_2]$ and $[v,x_1]$
satisfy the same relations defining $\mathcal S^{-1}\Lie^*$,
\item the Lie action and the commutative action satisfy the following 
mixed Leibniz relations:
\begin{equation}\tag{L1}
  v\cdot [x_1,x_2] = [v,x_1]\cdot x_2 - [v\cdot x_2,x_1]
  \end{equation}
\begin{equation}\tag{L2}
	[v,x_1x_2] =  [v\cdot x_1,x_2] + [v,x_1]\cdot x_2.
	  \end{equation}
\end{tenumerate}

\begin{theorem}
The operad $\PCalc$ is Koszul.
\end{theorem}

\begin{proof}
We use the distributive law criterion,
adapted to $2$-colored operads. 
Namely, the result that if $P$
and $P'$ are Koszul operads
and if $P''$ is a third operad obtained from a distributive
rule $P'\circ P \longrightarrow P\circ P'$, 
then $P''$ is itself Koszul~\cite{LV}*{Section 8.6}. We will apply this criterion
in the case $P=\Com^*$ and $P'=\mathcal S^{-1}\Lie^*$.

To do this orient the Leibniz rule (L1) to display it in the form
\[ [v\cdot x_2,x_1] =  [v,x_1]\cdot x_2 -v\cdot [x_1,x_2]
 \]
and rewrite the second Leibniz rule (L2) incorporating (L1) as follows:
\[[v,x_1x_2] =   [v,x_2]\cdot x_1 -v\cdot [x_1,x_2] + [v,x_1]\cdot x_2.  \]
In this way, along with the usual Leibniz rule, we have defined a map
\[ V\circ_{(1)} V' \longrightarrow V'\circ_{(1)} V \]
on the generators $V$ of $\Com^*$ and $V'$ of $\mathcal S^{-1}\Lie^*$
which, in turn, induces a surjective map
\[ \Com^* \circ \mathcal S^{-1}\Lie^* \longrightarrow \PCalc.\]
It is well known that, on the sub-operad where all outputs and
inputs are of the same (white) color, this is indeed an
isomorphism, so it suffices we check that this on the 
components of mixed color.

It is straightforward to check that $(\Com^* \circ \mathcal S^{-1}\Lie^*)(n)$
is of dimension $2\cdot n!$, so it suffices we check
that $\PCalc(n)$ is of the same dimension. To do this, we
recall that in~\cite{TTD} the authors provide us with a geometrical 
model of $\PCalc$.
Indeed, their results show there exists a topological operad $\mathsf{PreCyl}$ along with homotopy equivalences
\[\mathsf{PreCyl}(\circ;n,0)\simeq \operatorname{Conf}_n(\mathbb R^2)
 \quad \mathsf{PreCyl}(\bullet;n-1,1)\simeq \operatorname{Conf}_{n-1}(\mathbb R^2\smallsetminus 0), \]
and a surjection $\PCalc\longrightarrow H_*(\mathsf{PreCyl})$,
 and the Poincar\'e series of these spaces are known~\cite{Arnold}.
 This gives us the requisite lower bound on the dimensions of $\PCalc$,
 and shows that $\PCalc$ is obtained from a distributive law
from two Koszul operads. This shows it is Koszul,
as we claimed.
\end{proof}

\begin{remark}
Observe that this dimension count along with the methods
developed in~\cite{KhaKho} and the `quantum order'
of~\cite{Dotsenko2020}, allows us to exhibit a quadratic Gr\"obner
basis of $\PCalc$ where the leading terms of the mixed Leibniz rules
(L1) and (L2) are $[v\cdot x_1,x_2]$ and $[v,x_1x_2]$, as it was
done for the Poisson operad in~\cite{Dotsenko2020a}, for
the pre-Poisson operad in~\cite{Dotsenko2020}, and the for
Lie--Rinehart operad in~\cite{KhaKho}.
\end{remark}

Let us now
add a square zero operator $\delta$ of degree $1$ to $\PCalc$ along with
the non-quadratic relation encoding the magic
formula of Cartan
\[\textcolor{black}{  \delta(v\cdot x)-  \delta v \cdot x= [v,x]. }\]
We write this operad $\Calc$; it is the operad
 governing Tamarkin--Tsygan calculi. If we consider
 the increasing filtration on $\Calc$
  by the number of uses
 of $\delta$, we obtain an associated quadratic
 operad  $q\Calc$ governing $\PCalc$-algebras
 with a degree $-1$
 square zero operation $\delta$ such that
 \[\textcolor{black}{ [\delta v,x] = \delta[v,x],\qquad \delta(v\cdot x) = \delta v\cdot x.} \]
 Let us consider the (colored)
 operad of dual numbers
 $\D = T(\delta)/(\delta^2)$
 put in color~$\bullet$. Then we can present $q\Calc$ through
 a distributive rule $\lambda$ between $\PCalc$ and
 $\D$, and we claim that this is a distributive law. To do this, we need to check that
  \[ \rho \colon \D\circ \PCalc  \longrightarrow \PCalc\vee_\lambda \D =q\Calc  \]
  is an isomorphism.

  \begin{theorem}\label{thm:inhomog} The operad $q\Calc$
  is Koszul and $\Calc$ is inhomogeneous Koszul. Moreover, the map
  $q\Calc \longrightarrow \mathrm{gr}\,\Calc $
    is an isomorphism of symmetric sequences,
  so that   the underlying collections
  of $\Calc$ and $q\Calc$ are isomorphic to
  $\D\circ \PCalc$.
  \end{theorem}

\begin{proof}
Since $\D$ and $\PCalc$ are both Koszul, we deduce
that $q\Calc$ is also Koszul by a another use of the
distributive law argument of~\cite{LV}*{Section 8.6}.
To apply it, we need to check that $\rho$ is an isomorphism,
which
is immediate, since it is surjective and by~\cite{TTD}
the dimensions of the domain and the codomain match. This implies that $\Calc$ is
inhomogeneous Koszul, and hence the
Poincar\'e--Birkhoff--Witt theorem for
Koszul operads shows that
 $q\Calc \longrightarrow \mathrm{gr}\,\Calc$
is an isomorphism.
  \end{proof}

From this, we obtain a model
$\Calc_\infty$ of $\Calc$ with generators
\[ \Calc^\antishriek = \PCalc^\antishriek\circ \D^\antishriek\]
through the methods of~\cite{LV}*{Chapter 7}. Note that $\D^\antishriek = T[u]$ is a polynomial algebra generated by $u=\delta^\vee$, and when we write $\Calc^\antishriek$
we are doing it in the sense of inhomogeneous
Koszul duality theory.


\begin{remark}
In~\cite{TT} the authors state
that $\Calc$ is Koszul and give a description of what they refer to as $\Calc_\infty$-algebras.
However, it is not clear that $\Calc$ admits a
quadratic presentation, although it does admit
a quadratic-cubic presentation (where
the quadratic-cubic equations only
appear in the operations of mixed color). 
\end{remark}

\section{Main results}\label{sec:3}

Throughout, we fix an associative algebra $A$
concentrated in homological degree $0$ and
a model $B$ of $A$ (see Section~\ref{sec:1}),
and show how to compute the Tamarkin--Tsygan
calculus of $A$ using only the dga algebra $B$
and two small complexes $\X^*(B)$
and $\T_*(B)$ built out of $B$, which we like
to think of as the complexes of non-commutative
(nc) polyvector fields and non-commutative
differential forms
on $B$. We begin by recalling the essential details
of this algebraic structure we want to compute.

\subsection{Tamarkin--Tsygan calculi}\label{sec:TTC}

We have already introduced the operad
$\Calc$. It turns out the pair
\[ (\HH^*(A),\HH_*(A)) \]
admits a structure of a $\Calc$-algebra
 consisting of the
cup product, the cap product, Gerstenhaber's bracket  and Connes' operator. The Lie action on $\HH_*(A)$
is defined so that for each
$\omega\in \HH_*(A)$ and each $X\in\HH^*(A)$,
\[ L_X(\omega) = di_X(\omega)-\textcolor{blue}{(-1)^{|X|}}i_X(d\omega),\]
where we are using the notation $i_X$ for the cup
product action of $X$ on $\HH_*(A)$; this will be
convenient later.

That these operations
define a $\Calc$-algebra structure
was originally proved in~\cite{Daletskii}. See also~\cite{TT2}. We call this
algebra--module
pair $(\HH^*(A),\HH_*(A))$ along
with the operator $d$ the \emph{Tamarkin--Tsygan calculus}
of $A$.
We will write
 $\Calc_A$ for this $\Calc$-algebra,
 hoping it does not give rise to any confusion, in that
  we are not evaluating $A$ on
 the endofunctor $\Calc$, for example.

\begin{proposition}\label{proposition:TTC}
Let $A$ be an associative algebra. Then
$(\HH^*(A),\smile,
[-,-])$ is a Gerstenhaber algebra,
and $(\HH_*(A),i)$ is a module over
$(\HH^*(A),\smile)$ and $(\HH_*(A),L)$ is a module over $(\HH^*(A),[-,-])$.
Moreover, if for each $X \in \HH^*(A)$ we write
$i_X$ for the action of
$X$, for $Y\in \HH^*(A)$,
\[ [i_X,L_Y] = i_{[X,Y]} \quad\text{ and }\quad L_{X\cdot Y} =L_Xi_Y+{(-1)^{|X|}i_XL_Y}.\]

\end{proposition}

We
can describe all the operators of Proposition~\ref{proposition:TTC}
on the cochain level using the pair of classical complexes
$(C^*(A),C_*(A))$ as follows. For cochains
$f,g \in C^*(A)$
homogeneous of degrees $p$ and $q$ and a chain \[ z=a[a_1\vert\cdots\vert
a_{p+q}]\in C_*(A)\] of degree $n=p+q$, the cup product, the cap product, the Lie bracket and Connes' differential are defined by the following formulas,
where $\circ$ is Gerstenhaber's pre-Lie bracket
\[ f\circ g = \sum_{i=1}^n f\circ_i g\] on the
{endomorphism} operad $C^*(A)=
\End_A$ (see the Section~\ref{sec:2} for details):
\begin{titemize}
\item Cup product: $
	f \smile g=
		\mu \, (f\tt g) \, \Delta,$
\item Cap product: $i_f(z) =
	a f[a_1\vert\cdots\vert a_p]
		[a_{p+1}\vert \cdots\vert a_{p+q}],$
		\item Lie bracket: $
			[f,g] = f\circ g - (-1)^{(p-1)(q-1)} g\circ f$,
			\item Connes' differential: $ dz = \sum_{j=0}^n
	 (-1)^{jn}
		[a_{j+1}\vert\cdots \vert a_n \vert a \vert
		a_1 \vert\cdots\vert a_j] ${}{,}
		\item Lie action: $L_f(\omega) = [d,i_f](\omega)${}{.}
						 \end{titemize}
In other words, the homology of
 the pair of complexes
   \[ (C^*(A),C_*(A))\]
along with the operators $(\smile,[-,-],i,L,d)$
recovers the Tamarkin--Tsygan calculus
of the associative algebra $A$.

The work of Keller~\cite{DIH,KellInv1}, and later of
Keller and Armenta~\cite{KellArm,Calc} shows that the
Tamarkin--Tsygan calculus of an algebra is derived invariant.
Although our main result does imply it is homotopy invariant,
in the sense it induces a well-defined functor (as in Theorem 8.5.3 in~\cite{Hinich})
\[ \Ho(\Alg)^{\rm{iso}} \longrightarrow
 	\Ho(\Calc)^{\rm{iso}}, \]
invariance is not a novel result, since homotopy equivalent
algebras are derived equivalent. Rather, our main result is
more concerned with the explicit computation of this calculus
using a choice of resolution in $\Alg$.

To be precise, the result of Armenta--Keller implies that
one may \emph{attempt} to compute the calculus of an
associative algebra $A$ by choosing some model $B$ of it
and, in some way or another, obtain a description
of the spaces
$(\HH^*(A),\HH_*(A))$
and formulas for the $\Calc$-algebra structure here.
Knowing this, our goal here is, having chosen~$B$,
to carry out the last step as explicitly as possible,
and to give a datum depending on~$B$ that makes such
computation feasible.

\subsection{Fields and forms}\label{sec:ncthings}

Let us consider the $B$-bimodule resolution
of $B$ given by
\[ S_*(B) {}{\colon} 0\longrightarrow	
	B\otimes V\otimes B\longrightarrow
	 B\otimes B\longrightarrow 0,\]
where the only non-trivial map is given
by \[ b\otimes v\otimes b' \longmapsto
bv\otimes b' - b\otimes vb'.\]
In case $B$ has no differential, this is the usual small
resolution of the free algebra $TV$. In our case, there
is a caveat: this is a semi-free resolution, in the sense
that
\[ B\otimes V\otimes B\]
 can be filtered by
sub-bimodules (using the triangulation of $V$) in such
a way that the successive quotients are free $B$-bimodules
on a basis of cycles. Indeed, once we have considered the filtration on
$B\otimes V\otimes B$ using a triangulation of $V$, the differential
that is nontrivial on $V$ will vanish. Since
this is precisely the
kind of resolutions
needed to compute
Hochschild (co)homology
in the dg setting, we may
use this resolution
to proceed with our computations.

To understand
where the differential of $B\otimes V\otimes B$
of comes from, it is
useful to note it
identifies with the
bimodule of
\emph{Kahler differentials $\Omega_B^1$} of $B$.
This is the free $B$-bimodule on $B$, say spanned
by basis elements $
b\otimes b'\otimes b''$
subject to the relations
\[ b' \otimes  b_1b_2\otimes b''= b' b_1 \tt b_2\tt b''+b' \tt b_1 \tt b_2b''.\]
Note that since $B$ is
generated by $V$, any class in $\Omega_B^1$,
which we write $bdb'b''$,
can be written in the form $b_1 dv b_2	$, and the
identification is such that
\[ g{}{\colon} b\otimes v\otimes b' \longmapsto b dv b'. \]
The following lemma makes
this precise.
\begin{lemma}\label{lema:smallV}
The map
$g{}{\colon} B\tt V\tt B
 \longrightarrow
  \Omega_B^1$
  is an isomorphism of $B$-bimodules.
\end{lemma}

\begin{proof}
The $B$-bimodule $\Omega_B^1$ satisfies the
universal property that
the restriction map induced
from $B\longrightarrow \Omega_B^1$ sending $b\longmapsto db$ induces
a bijection
\[ \hom_{B^e}(\Omega_B^1,-)
 \longrightarrow
 	\Der(B,-). \]
Since $B$ is free on $V$,
this last functor is
naturally isomorphic to
the functor $\hom(V,-)$.
 It is immediate that
 the inclusion of $V$
 into $B\otimes V\otimes B$
 satisfies the analogous universal property, which
 gives the result.
\end{proof}

Since
$S_*(B)$ gives us
a resolution of $B$
in $B$-bimodules
the complexes
\[ \X^*(B) =
	\hom_{B^e}(S_*(B),B) ,
\quad
	\T_*(B) =
	B \tt_{B^e} S_*(B) \]
 compute Hochschild
(co)homology of $B$ and
hence, by invariance,
that of $A$. We call
$\X^*(B)$ the \new{space of nc poly vector fields on $B$},
and $\T_*(B)$ the \new{space of nc differential forms on
$B$}. Let us now observe
that from Lemma~\ref{lema:smallV} it follows
these complexes take
a very simple form,
which also explains the
origin of their names.

\begin{corollary}
We have natural identifications
 of complexes
\[ \X^*(B) =
 \cone(\mathrm{Ad}{}{\colon} B\longrightarrow \Der(B)),
 \quad
\T_*(B) = \cone(\mathrm{Co}{}{\colon} B\otimes V \longrightarrow B).
\]

\end{corollary}
Here, the map $\mathrm{Ad}$
is the adjoint map of $B$
and the map $\mathrm{Co}$ is
the commutator map of $B$.

{{}}\label{fourterm}
It is worthwhile to observe that
since $H_*(B) = H_0(B)= A$, there is
a four term exact sequence coming for the
long exact sequence of the cone
\[ 0\longrightarrow \HH_1(A) \longrightarrow H_0(B\otimes V)
	\longrightarrow H_0(B) \longrightarrow \HH_0(A)
	\longrightarrow 0 \]
and it is straightforward to check that the image of the middle arrow
is
\[[A,A]\subseteq A = H_0(B),\] which recovers the usual
description of $\HH_0(A)$. Moreover, we see that $\HH_1(A)$
is the kernel of this map, and that for $n\in\NN_{\geqslant 2}$, we an identification
$\HH_n(A) = H_{n-1}(B\otimes V)$.
Dually, for $n\in\NN_{\geqslant 2}$ we have an
identification
\[\HH^n(A) =  H^{n+1}(\Der(B)),\] and a four term exact sequence
\[0\longrightarrow\HH^0(A) \longrightarrow H^0(B)
	\stackrel{j*}{\longrightarrow} 	H^0(\Der(B))
		\longrightarrow \HH^1(A) \longrightarrow 0,\]
that shows that $\HH^1(A) = H^0(\Der(B))/ \mathrm{im}(j^*)$.
A far-reaching generalization of this, explaining the relation
of operadic cohomology and Hochschild cohomology of operads
under reasonable homotopical hypotheses is present in~\cite{Rezk}*{Theorem 1.3.8}.

We now show that the resolution
$S_*(B)$ is retraction of the double
sided bar resolution of $B$, a fact
known since the inception of Hochschild
cohomology. Once we have developed the
homotopy theory of calculi, this will
give us a way of transporting the homotopy
coherence calculus of chains on $B$ onto
the pair $(\X^*(B),\T_*(B))$. For convenience,
we recall that the data of a homotopy
deformation retract
is a triple of maps of complexes  $(i,\pi,h)$
\[ i{}{\colon}C'\longrightarrow C,
	\quad
 \pi {}{\colon} C\longrightarrow C' \]
such that $\pi i =1$ and $i\pi-1 = dh+hd$.

\begin{lemma}\label{lema:retract} There is a homotopy deformation retract $(i,\pi,h)$ between $S_*(B)$ and
the double sided bar resolution $\mathsf{Bar}_*(B)$
of $B$, where we may take
\[
\pi {}{\colon} \mathsf{Bar}_*(B) \longrightarrow S_*(B)\]
 to be either zero or the identity
except on $B\otimes B\otimes B$ where
for $b = v_1\cdots v_n$ we have
\[ \pi(1\otimes b\otimes 1) =
	\sum_{i=1}^n v_1\cdots v_{i-1}\otimes
		v_i\otimes v_{i+1}\cdots v_n.\]
In particular, the chain complex $C_*(B)$ is homotopy equivalent to $\T_*(B)$ and the cochain complex $C^*(B)$ is
homotopy equivalent to~$\X^*(B)$.
\end{lemma}

\begin{proof}
Let us begin by noting that
the non-trivial component
of $\pi$ identifies with
the quotient
map $q {}{\colon} B\otimes B\otimes
B \longrightarrow
\Omega_B^1$
under the isomorphism of
Lemma~\ref{lema:smallV}, so it is
immediate it is a map of complexes.
To construct $i$, we observe that
if $B$ has zero differential, we
may choose it to be the inclusion
of $B\otimes V\otimes B$ into
$B\otimes B\otimes B$.

In case
$B$ has a differential, we can
filter the total complexes above
using its columns, and then the
extra differential coming from
$B$ lowers this filtration degree.
In this way, the homological perturbation
lemma guarantees that we can obtain
again a homotopy retract.
\end{proof}

It would be desirable to have
an explicit formula of the maps $i$ and $h$.

\subsection{The homotopy calculus}\label{subsec:MMC}

Having already
set up things to obtain a  retract from Hochschild
(co)chains of $B$ onto the pair
$(\X^*(B),\T_*(B))$, we recall that
{Hochschild}{} (co)chains on $B$ carry
a homotopy coherent calculus structure
(which we can understand at the `strict'
level) and use this to obtain
a homotopy coherent calculus structure on
$(\X^*(B),\T_*(B))$ whose homotopy type
recovers $\Calc_A$.

Let us recall from the Section~\ref{sec:2} that
there is a $2$-colored operad $\PCalc$ so that
a $\PCalc$ algebra is the same as a pair $(V,M)$
where $V$ is a Gerstenhaber algebra and
$M$ is a $V$-module. As explained there,
extending the operad
$\PCalc$ by a square zero
unary operation $d$ of the second {color}{}
and imposing the Cartan formula gives us
the operad $\Calc$.

We proved that the
operad $\PCalc$ is quadratic Koszul, while
the operad $\Calc$ is inhomogeneous Koszul.
In particular, $\Calc$ admits a model where the
only non-quadratic differential comes from
the Cartan formula and, in particular, there is
$\Gers_\infty$ placed in the first {color}{}. We will write $\Calc_\infty$
for this model, although it is not the same model
that the authors consider in~\cite{TTD}. With
this at hand, we recall from~\cite{TTD}*{Section 4.3}
the following technical results:

\begin{theorem}
There is a dg colored operad $\mathsf{KS}$,
the Kontsevich--Soibelman operad,
and a topological colored operad $\mathsf{Cyl}$
such that
\begin{tenumerate}
\item the operad $C_*(\mathsf{Cyl})$ is formal
and its homology is $\Calc$,
\item there is a quasi-isomorphism $\mathsf{KS}
\longrightarrow C_*(\mathsf{Cyl})$ of
dg-operads,
\item $\mathsf{KS}$ acts on the pair
$(C^*(-),C_*(-))$ in such a way that
\item on homology we obtain the
usual $\Calc$-algebra structure on
$(\HH^*(-),\HH_*(-))$.
\end{tenumerate}
In particular, for every cofibrant replacement $Q$ of $\Calc$, the pair $(C^*(A),C_*(A))$ is a
$Q$-algebra which on homology gives $\Calc_A$.
\end{theorem}

\begin{proof}
Claims (1){}{--}(4) are in the cited references, and
the last claim is standard, but let us explain
it, since we will apply it for our choice of
cofibrant replacement of $\Calc$. That $Q$ is a
cofibrant replacement of $\Calc$ means that there
is a quasi-isomorphism $Q\longrightarrow \Calc$
and $Q$ is a cofibrant operad in the model category
of dg operads. Since we have a surjective
quasi-isomorphism $\mathsf{KS}\longrightarrow \Calc$ or, what is the same,
a trivial fibration. But $Q$ is cofibrant,
so we can obtain a lift $Q\longrightarrow \mathsf{KS}$. Finally, since
 $\mathsf{KS}$ acts on the pair $(C^*(A),C_*(A))$, so does $Q$, and on homology we obtain $\Calc_A$.
\end{proof}

In particular, from the
theorem it follows that $(C^*(B),C_*(B))$ admits a $\Calc_\infty$-algebra
structure which on homology gives
the $\Calc$-algebra $\Calc_B=\Calc_A$, and
we will write $\Calc_{\infty,B}$ for
this structure. We point out that the explicit construction of the operads $\mathsf{KS}$ and
$\mathsf{Cyl}$ are not necessary for us here.

We remind the reader that $\Calc_\infty$
is our choice of cofibrant replacement obtained
through the inhomogeneous Koszul duality theory.
Our main result is that the pair
$(\X^*(B),\T_*(B))$ admits
$\Calc_\infty$-algebra structure which on homology
gives the usual $\Calc$-algebra structure; for
some formulas see Theorem~\ref{thm:formulas}.

\begin{theorem}
For any quasi-free model $B$ of $A$,
the pair $(\X^*(B),\T_*(B))$ admits
a $\Calc_\infty$-algebra
structure which is $\infty$-quasi-isomorphic to the
$\Calc_\infty$-algebra $\Calc_{\infty,A}$.
In particular, this structure
recovers $\Calc_A$ by taking homology.
\end{theorem}

\begin{proof}
By results of Section~\ref{sec:1}, the operad $\Calc$
is inhomogeneous Koszul. The methods developed
in~\cite{BV} imply that $\Calc$-algebras then
have the same rich homotopy theory that algebras
over Koszul operads do. In particular by
\cite{BV}*{Theorem 33}, any $\Calc_\infty$-algebra
structure may be transported through a homotopy equivalence
of chain complexes. By Lemma~\ref{lema:retract}, there
is a homotopy equivalence
\[ (C^*(B),C_*(B))
	\longrightarrow (\X^*(B),\T_*(B)),\]
which implies the $\Calc_\infty$-algebra structure
on $(C^*(B),C_*(B))$ may be transported to one
on the pair $(\X^*(B),\T_*(B))$.  The
fact this structure is $\infty$-quasi-isomorphic
to that one on 
$\Calc_{\infty,A}$ follows by another use
of the transfer theorem, this time using
$\Calc_{\infty,A}$ and $\Calc_{\infty,B}$.
\end{proof}

The space $\X^*(B)$ of nc poly vector
fields on $B$ is a dg Lie
algebra under the usual bracket between derivations, with
differential $[\partial_B,-]$. We can now state the next theorem, which
in particular tells us that to compute Gerstenhaber brackets in $A$ we may
do so by choosing any model $B$ of $A$ and computing the usual
Lie bracket in $\X^*(B)$. This last statement is a special
case of~\cite{Hinich}*{Theorem 8.5.3}.

\begin{theorem}\label{thm:formulas}
Let $B=(TV,d)$ be a quasi-free model of $A$. Then
\begin{tenumerate}
\item the cup product can be computed through
the symmetrization of dual cobrace operation
$x_1x_2 = \{x_1,x_2;d\}$ of Theorem~\ref{thm:cobraces},

\item the Gerstenhaber bracket is the Lie bracket
on the semi-direct product of a shifted copy of $B$ with its Lie algebra of
derivations,
\item the Connes boundary can be computed by
 \[ d\omega = \sum_{i=1}^n (-1)^\varepsilon
	v_{i+1}\cdots v_nv_1\cdots v_{i-1}
	dv_i \]
for $\omega = b+b'dv$ be a differential
form in $\T_*(B)$  and $b = v_1\cdots v_n$.
\end{tenumerate}	
\end{theorem}

\begin{proof}
In Theorem~\ref{thm:cobraces} we will see that
the cup product above is part of a brace algebra
structure on $\X(B)$ which has the same
quasi-isomorphism type as the brace algebra $C^*(B)$,
which implies that the symmetrization of the operation in that formula gives the desired cup product in cohomology. Similarly, the Lie bracket
corresponds to the antisymmetrization of the brace
operation $\{x_1;x_2\}$ which is just the composition. Finally, the claim for the Connes
boundary can actually be checked directly, using,
for example, the definition of the boundary map
of the usual LES of {Hochschild}{} and cyclic cohomology.
\end{proof}

\section{Other results and computations}\label{sec:4}

\subsection{Brace operations}

Let $V$ be a graded vector space. A
$B_\infty$-algebra structure on $V$ is the datum
of a structure of dg bialgebra on $T(sV)$ where
the comultiplication is given by deconcatenation~\cite{GetJon}.
It follows that the data required to define such
structure amounts to a differential on $T(sV)$,
which gives $V$ the structure of an $A_\infty$-algebra,
along with a multiplication on $T(sV)$. The fact
this is a map of coalgebras means it is completely
determined by a map
$T(sV)\otimes T(sV)\longrightarrow sV$. This defines, for each
$(p,q)\in\NN\times \NN$, a map of degree $1-p-q$
\[ \mu_{p,q} {}{\colon}V^{\tt p}\otimes
V^{\tt q} \longrightarrow V.\]

Let $B$ be a quasi-free dga algebra. We
proceed to show that its space of nc-vector fields $\X^*(B)$
admits a $B_\infty$-algebra structure.
We have already noted it admits an
$A_\infty$-structure, so it suffices we define the
family of maps corresponding to the multiplication. As it happens for Hochschild cochains,
we can arrange it so that for each $(p,q)\in\NN\times \NN$ we have $\mu_{p,q} = 0$ whenever
$q>1$, in which case what we have
is an algebra over the operad of
braces $\Br$.

For linear maps $f_1,\ldots,f_n,g\in \hom(V,B)$
corresponding to a derivation in $\X^*(B)$, we define the operation
$\{f_1,\ldots,f_n;g\}$ as follows. Let $\mathrm{sh}(f_1,\ldots,f_n)$
be the unique derivation on $B$ that acts by zero on
monomials of length less than $n$, and acts, for
$k\in\NN$ on monomials of length $n+k$ by the sum
\[\sum_\sigma \sigma(f_1\otimes\cdots\otimes f_n\otimes 1^{\tt k})\] as $\sigma$ runs through $(n,k)$-shuffles
in $S_{n+k}$. In other words, each $\sigma$ keeps $f_1,\ldots,f_n$
in order and puts the identities in the middle slots. If $G$ is the derivation that corresponds
to $g$, we set $\{f_1,\ldots,f_n;g\}$ to be the linear
map corresponding to the derivation $\mathrm{sh}(f_1,\ldots,f_n)\circ G$. We call $\{f_1,\ldots,f_n;g\}$ a \emph{dual brace
operation}.

Recall from~\cite{BraceOp} that if $A$ is an associative
algebra, there are \emph{brace operations} defined
on $C^*(A)$ that make it, along with its usual
structure of a dga algebra, into a
brace algebra.
Concretely, for each $n\in\NN$ and for
$f,g_1,\ldots,g_n\in C^*(A)$ we have that
\[ \{f;g_1,\ldots,g_n\} = f\circ  \mathrm{sh}(g_1,\cdots,g_n).\]
 In other words, we are inserting
$g_1,\ldots,g_n$ into $f$ in all possible ways
preserving their order. For example, $\{f;g\}$ is
the circle product of Gerstenhaber whose antisymmetrization
gives the Gerstenhaber bracket. Note our notation
is not ambiguous here: $\{f;g\}$ is both a brace
and a dual brace operation, in both cases it
is induced by composition of (co)derivations.

\begin{theorem}\label{thm:cobraces}
Let $B = (TV,d)$ be a quasi-free dga algebra.
The dual braces give the space of nc fields
$\X^*(B)$ the structure of a $\Br$-algebra. In particular, for every quasi-free algebra $B$, the
space of nc fields on $B$ is a $\Gers_\infty$-algebra. Moreover, $\X^*(B)$ is $\Br$-quasi-isomorphic to the brace algebra $C^*(A)$.
\end{theorem}

\begin{proof}
The first part of the first claim is proved,
mutatis mutandis, in the same way one shows
brace operations on $C^*(A)$ satisfy the
defining equations of a brace algebra: the
definitions are dual to each other. The second claim
follows from a deep result of Tamarkin~\cite{HinichTam},
which shows that there is a quasi-isomorphism of operads
$\Gers_\infty \longrightarrow \Br$. In this way,
one can obtain from the brace algebra structure of $\X^*(B)$ to a $\Gers_\infty$-structure.

For the last claim, we note that $C^*(B)$ is
$\Br$-quasi-isomorphic to $C^*(A)$. Our theorem
implies that $\X^*(B)$ is, in particular,
$\Gers_\infty$-quasi-isomorphic to $C^*(B)$
and, by Tamarkin's result, we deduce that these
two are also $\Br$-quasi-isomorphic.
\end{proof}

\textit{A word of caution.}
The cup product
we have defined using the differential
$d$ of $B$ is not commutative on the nose,
and in fact the cup product that one obtains using
the result of Tamarkin is given by the symmetrization
of this product. Since the cup product is commutative
when passing to homology, there is no harm in using
the non-symmetrized version for computations.

\subsection{Duality}

In~\cite{Hers}, Herscovich fixes an augmented
weight graded connected dg algebra $A$ and proceeds to
show that, writing $E_A$ for the dual of the dg coalgebra
$BA$, the Tamarkin--Tsygan calculi of $A$ and $E_A$
are dual to each other.

In the same context, let us
pick a quasi-free model $(TV,d)=B\longrightarrow A$
of $A$, where $V$ is a weight graded
$A_\infty$-coalgebra. It then makes sense to
consider the Tamarkin--Tsygan calculus of the
$A_\infty$-algebra $E_A = (sV)^\#$. From our main theorem,
one obtains the, now tautological, extension of Herscovich's result.

 We also point out
that, following Berglund and B\"{o}rjeson~\cite{Alexander},
the algebra $A$ is $A_\infty$-Koszul with Koszul dual $E_A$. Their result also
implies that $\X^*(B)$ is $A_\infty$-quasi-isomorphic to the {Hochschild}{} dg-algebra
$C^*(A)$ with its canonical product and
differential.

\begin{theorem}
The Tamarkin--Tsygan calculi of $A$ and $E_A$
are dual to each other in the sense that there is
an isomorphism of Gerstenhaber algebras
and an isomorphism of Gerstenhaber modules,
\[ f{}{\colon} \HH^*(E_A)\longrightarrow \HH^*(A)\quad
 g{}{\colon}\HH_*(A)^\#\longrightarrow
 	\HH_*(E_A)\]
respectively,
 	where $\HH_*(A)$
 	is a module over $\HH^*(E_A)$ through $f$. Moreover, we have that $\delta_E g=-g\delta_A^\#$ where $\delta_E$ is the Connes boundary of
 	$\;\HH_*(E_A)$ and $\delta_A$ is that of
 	$\;\HH_*(A)$.
 	
\end{theorem}

\begin{proof}
We already know that to compute the Tamarkin--Tsygan
calculus of $A$ we may use nc-fields and nc-forms
obtained from $B$. To compute the Tamarkin--Tsygan
calculus of $E_A$, we may use a quasi-free dg model
of the dg coalgebra $BE_A$ ---the
$\infty$-bar construction on $E_A$--- and nc-cofields,
that is, coderivations, and nc-coforms on it: this
is just the classical definition. These two constructions
are dual to each other: the space
$\hom\left(BE_A,E_A\right)$
is equal to $\hom(V,TV)$ while
$BE_A\otimes E_A$ is dual to $V\otimes TV$,
and as explained in~\cite{Hers}, these isomorphisms are compatible with
the cup and cap products, the Lie bracket, and
Connes' boundary.
\end{proof}

\subsection{Examples of computation}\label{sec:compute}
We now give two examples where, in the spirit of~\cite{FMT}, we compute
Hochschild cohomology of an algebra $A$ using a (minimal) model of it. We will also compute, in some cases, Hochschild homology
and cyclic homology, and the action of Hochschild cohomology on
Hochschild homology using the results of the previous sections. We remark that the computations
carried out here are many and for convenience we
have omitted them. However, they are all simple
and straightforward computations with no
intricacies whatsoever, so nothing should be
lost to this omission.

\medskip

\noindent\textit{A crown quiver algebra.}
Let us consider the quiver as in Figure~\protect\ref{wheelfigure}
with the single relation $\alpha_1\alpha_2\cdots\alpha_r\alpha_1$, and its associated algebra $A$.
In~\cite{ReRo}, the authors compute its Hochschild cohomology, including
the Gerstenhaber bracket and the cup product. We will recover their
results using the minimal model of $A$.

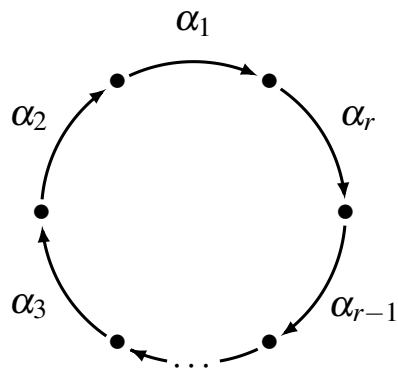
\begin{figure}[ht]
\centering
\vspace{-4 pt}
\begin{tikzpicture}[scale = 0.6]
    \foreach \a in {1,2,3,4,6}
    {
        \node (u\a) at ({\a*60}:2){$\bullet$};
        \draw [latex-,
        		line width = 1.15 pt,
        		domain=\a*60-55:\a*60-5] plot ({2*cos(\x)}, {2*sin(\x)});
    }
    \node (a1) at ({1*60+33}:2.5){$\alpha_1$};
    \node (a2) at ({2*60+33}:2.5){$\alpha_2$};
    \node (a3) at ({3*60+33}:2.5){$\alpha_3$};
    \node (a4) at ({4*60+33}:2.05){$\cdots$};
    \node (a5) at ({5*60+33}:2.8){$\alpha_{r-1}$};
    \node (a6) at ({6*60+33}:2.5){$\alpha_r$};
    \node (u4) at ({5*60}:2){$\bullet$};
    \draw [latex-,	line width = 1.15 pt,
        		domain=245:260] plot ({2*cos(\x)}, {2*sin(\x)});
    \draw [line width = 1.15 pt,
        		domain=280:295] plot ({2*cos(\x)}, {2*sin(\x)});
\end{tikzpicture}
\caption{The wheel}\label{wheelfigure}
\end{figure}

We begin by noting that for $n\in\NN$
we have that $\Tor_A^{n+1}(\kk,\kk)$ is one dimensional
spanned by the class
of the chain
\[\varepsilon_n =[\alpha_1\vv \alpha_2\cdots\alpha_r \alpha_1\vv
\cdots\vv \alpha_2\cdots\alpha_r\alpha_1]\] while, as usual, $\Tor_A^1$ is
spanned by
the arrows. One way to see this is to argue that
the Anick resolution of the trivial $A$-module $\kk$,
which has the overlappings
\[ \tau_n = \alpha_1(\alpha_2\cdots\alpha_r\alpha_1)^n \]
corresponding to $\varepsilon_n$ as $A$-module
generators, is minimal, and hence these generate the corresponding $\Tor$-groups. The minimal model $B$ has then $r$ generators in degree~$0$,
and we write $\varepsilon_0$ the one corresponding to $[\alpha_1]$,
and for each $n\in\NN$ a generator $\varepsilon_n$ in degree $n$ whose
differential is, by the main result of~\cite{Tamaroff}, as follows:
\[\label{eqn:diffcrown} \partial(\varepsilon_n) = \sum_{s+t=n-1}(-1)^s \varepsilon_s\alpha_2\cdots\alpha_r\varepsilon_t.\]
We are intentionally suppressing the sign given by the binomial
coefficient since in this case there is exactly one non-vanishing
higher coproduct, $\Delta_{r+1}$.

To find $\HH^*(A)$, observe that for each
natural number $n\in\NN_0$ there is
an obvious cycle $f_n$ of degree $-n$ in $\Der(B,A)$ such
that $f_n(\varepsilon_n) = \varepsilon_0$, and one can check, as it is done
in~\cite{ReRo}, that it provides a generator for $\HH^{n+1}(A)$, which
is therefore one dimensional. We will now find a derivation of $B$
that covers $f_n$ under $\alpha{}{\colon}B \longrightarrow A$, and then
compute the Gerstenhaber bracket with these cycles: since the arrow
$\alpha_*$ is a quasi-isomorphism, we deduce that these cycles represent
generators for the cohomology groups of $\Der(B)$.

To do this, let us fix $n\in\NN_0$ and let $F$ be a derivation
of degree $-n$ such that $F(\varepsilon_n) = \varepsilon_0$. Recursively
solving for the values of $F$ on generators using the equation \[\partial F =
(-1)^n F\partial,\] shows that the following choice of
lift works:
\[ F_{2n}(\varepsilon_t) =  (t-2n+1) \varepsilon_{t-2n} ,\quad
 F_{2n+1}(\varepsilon_t) = \begin{cases}
 	\varepsilon_{t-2n-1}, & \text{for $t$ odd}{}{,} \\
 	0, & \text{for $t$ even}.
 		\end{cases} \]
With this choice of generators of $\HH^*(A)$, we compute that
\[
 [F_m,F_n] =
 	\begin{cases}
 		 (m-n)F_{m+n}, & \text{ for $m,n$ even,}\\
  		 0, & \text{ for $m,n$ odd,}\\
 		 m F_{m+n}, & \text{ for $n$ even, $m$ odd.}		
 	\end{cases}
 	\]
This coincides with the formulas obtained in~\cite{ReRo}. However,
observe that since we are using the natural grading in $\Der(B)$,
the derivations in odd degree represent elements of even degree in
$\HH^*(A)$, and those of even degree represent elements of odd
degree in $\HH^*(A)$, which explains the shift in our formulas.

Note that since $B$ has no quadratic part
in its differential, the cup product
structure in $\HH^*(A)$ is trivial. This means,
in particular, that the Gerstenhaber algebra structure on $\HH^*(A)$ is independent of the
parameter $r$, but one can check that the
higher products can be used to distinguish them:
the $A_\infty$-structures on Hochschild cohomology
are not quasi-isomorphic for distinct parameters,
which shows in particular that these algebras
are not derived equivalent; that is, their
derived categories are not equivalent.

We now observe that $\HC_*(A)$ is
easily computable: it is nothing but the homology
of the abelianization of $B$, and this has
a simple description.
Indeed, since $B$ is
quasi-free, the space $B/[B,B]$ is
spanned by equivalence classes of cyclic words in
$B$ with respect to cyclic shifts. Moreover,
the differential of $B$ in~\ref{eqn:diffcrown}
lands in $[B,B]$: this follows from
the fact that the non-cyclic arrow $\varepsilon_n$ lies in the commutator
subspace, and the differential preserves it, hence we deduce that $\HC_*(A)=B/[B,B]$.

\medskip

\noindent\textit{A non-$3$-Koszul algebra.}
Let us consider the following quiver $Q$ with relations $R = \{xy^2,y^2z\}$.
We will compute its minimal model and with it its Hochschild cohomology,
including the bracket. We will also compute the cup product;
since the coproduct on $\Tor_A(\kk,\kk)$ is non-vanishing only on the
generator which we call $\Gamma$, this computation is
 straightforward.

By the main result in~\cite{Tamaroff}, the algebra
$ A = \kk Q/(xy^2,y^2z)$ has minimal model~$B$ given by the free
algebra over the semisimple $\kk$-algebra of vertices $\uk$ with set of homogeneous generators
$ \{ x,y,z,\alpha,\beta,\Gamma,\Lambda \} $
such that
\begin{align*}
 &\partial x=\partial y=\partial z=0, \quad
	\partial \alpha = xy^2,\quad \partial \beta = y^2z, \quad
	\partial\Gamma = \alpha z-x\beta,\quad \partial\Lambda = xy\beta-\alpha yz. \end{align*}
		
\noindent Here $\Gamma$ corresponds to the overlap $xy^2z$ while
$\Lambda$ corresponds to the overlap $xy^3z$, and~$\alpha$
and $\beta$ {correspond}{} to the relations they cover under the
differential $\partial $ of the model, so that $x,y$ and $z$ are
in degree $0$, $\alpha$ and $\beta$ in degree $1$, and
$\Gamma$ and $\Lambda$ in degree~$2$. Thus~$B$ is the
path algebra of the dg quiver in Figure 2.

\begin{figure}[ht]
\centering
$\vcenter{\hbox{
\begin{tikzpicture}[scale=0.8]
\node (e3) at (0,0){$\bullet$};
\node (e2) at (2,0){$\bullet$};
\node (e1) at (4,0){$\bullet$};
\node (x) at (2.9,-.35){$x$};
\node (z) at (.9,-.35){$z$};
\node (y) at (2,1.40){$y$};
\draw[-latex,line width = 1.15 pt] (e3)--(e2);
\draw[-latex,line width = 1.15 pt] (e2)--(e1);
\draw[-latex,line width = 1.15 pt] (e2) to [out=60,
				in=120,
				looseness=10] (e2);
\end{tikzpicture}}}$
\qquad\qquad
$\vcenter{\hbox{
\begin{tikzpicture}[yscale=0.6]
\node (e3) at (0,0){$\bullet$};
\node (e2) at (2,2){$\bullet$};
\node (g) at (2,.35){$\Gamma$};
\node (l) at (2,-.85){$\Lambda$};
\node (e1) at (4,0){$\bullet$};
\node (x) at (3.8,1.25){$x$};
\node (z) at (.2,1.25){$z$};
\node (alpha) at (2.8,0.65){$\alpha$};
\node (beta) at (1.2,0.65){$\beta$};
\node (y) at (2,3.35){$y$};
\draw[-latex,line width = 1.15 pt] (e3) to[bend left = 25](e2);
\draw[-latex,line width = 1.15 pt] (e3) to[bend left = 0](e2);
\draw[-latex,line width = 1.15 pt] (e3) to[bend right = 20](e1);
\draw[-latex,line width = 1.15 pt] (e3) to[bend left = 0](e1);
\draw[-latex,line width = 1.15 pt] (e2) to[bend left = 25](e1);
\draw[-latex,line width = 1.15 pt] (e2) to[bend right =0](e1);
\draw[-latex,line width = 1.15 pt] (e2) to [out=60,
				in=120,
				looseness=10] (e2);
\end{tikzpicture}}}$
\caption{A quiver and its dg resolution with
respect to the relations $R$}
\end{figure}
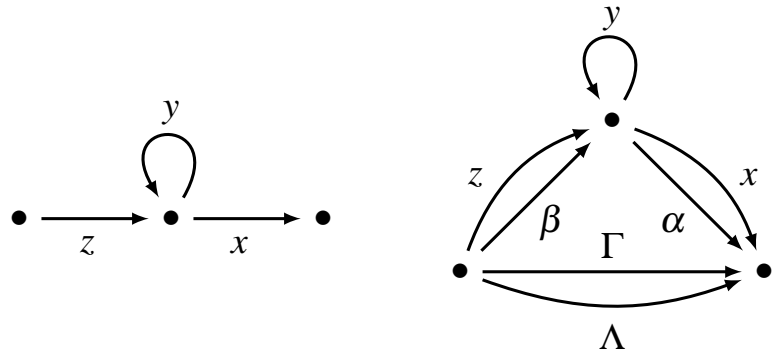

The elements in $B$ are the following,
where $r,t$ are elements of $\{0,1\}$ and $s\in\NN_0$:
\begin{itemize}
\item \emph{degree zero:} $x^ry^sz^t$,
\item \emph{degree one:} $\alpha y^s z^t,x^r y^t\beta$,
\item \emph{degree two:} $\Gamma,\Lambda,\alpha y^s \beta$.
\end{itemize}
Since we know that $\Tor_A^{\geqslant 4}(\kk,\kk)$ is zero, it follows
that $\HH^{\geqslant 4}(A)$ is zero. Moreover, it is straightforward
to see that $\HH^0(A) = Z(A)$ has basis
\[\{xyz,xz,y^2,y^3,
\ldots\},\] so we may focus our attention
on derivations of degree $0$, $-1$ and $-2$ to obtain bases for the remaining groups
$\HH^1(A)$, $\HH^2(A)$ and $\HH^3(A)$.

The following is a basis of derivations for the
$\uk$-bilinear $0$-cycles, where $s\in\NN_0$, and we adopt the
convention that $\alpha y^{-1}\beta = \Lambda$ and $\alpha y^{-2}
\beta=\Gamma$:
\begin{align*}
	&E_s(x) = 0,  &
	&E_s(y) =y^{s+1},  &
	&E_s(z) = 0,  &
	&E_s(\alpha) = 2\alpha y^s,\\
	&E_s(\beta) = 2y^s\beta{}{,} &
	&E_s(\Lambda) = 3\alpha y^{s-1} \beta, &
	&E_s(\Gamma) = -2\alpha y^{s-2}\beta, \\
	&F_s(x) = xy^s,  &
	&F_s(y) =0,  &
	&F_s(z) = 0,  &
	&F_s(\alpha) = \alpha y^s,\\
	&F_s(\beta) = 0{}{,} &
	&F_s(\Lambda) = \alpha y^{s-1} \beta, &
	&F_s(\Gamma) = -\alpha y^{s-2}\beta, \\
	&G_s(x) = 0,  &
	&G_s(y) = 0 ,  &
	&G_s(z) = y^sz,  &
	&G_s(\alpha) = 0 ,\\
	&G_s(\beta) = y^s\beta, &
	&G_s(\Lambda) = \alpha y^{s-1} \beta, &
	&G_s(\Gamma) = -\alpha y^{s-2}\beta.
\end{align*}

We now compute the $\uk$-bilinear $0$-boundaries. A basis for
them is given by the following family of derivations, where $s\in\NN_0$:
\begin{align*}
	&T_s(x) = xy^{s+2}, &
	&T_s(y) = 0, &
	&T_s(z) =0,&
	&T_s(\alpha) = \alpha y^{s+2}, \\
	&T_s(\beta) = 0, &
	&T_s(\Lambda) = \alpha y^{s+1}\beta, &
	&T_s(\Gamma) = -\alpha y^s \beta, \\
	&R_s(x) = 0, &
	&R_s(y) = 0, &
	&R_s(z) = y^{s+2}z, &
	&R_s(\alpha) = 0, \\
	&R_s(\beta) = y^{s+2}\beta, &
	&R_s(\Lambda) = -\alpha y^{s+1}\beta, &
	&R_s(\Gamma) = \alpha y^s\beta{}{.}
\end{align*}
By direct inspection, we have that $F_{s+2} = T_s,G_{s+2} = R_s$ for
$s\in\NN_0$, and no other relations, so that $H^0(\Der(B))$
has infinite dimension and is spanned the classes of the elements in the
set $\{F_0,F_1,G_0,G_1,E_s : s\in \NN_0\}$. Moreover, a basis for
$H^0(B)$ is of course given by the monomials belonging to $A$.
We may only worry about the adjoint action on $B$ of $y$
and $e_2$, since they are the only ones that are $\uk$-bilinear.
We have that
\begin{align*}
	&\Ad_y(x) = -xy, &
	&\Ad_y(y) = 0, &
	&\Ad_y(z) = yz, \\
	&\Ad_y(\alpha) = -\alpha y, &
	&\Ad_y(\beta) = y\beta, &
	&\Ad_y(\Lambda) =0, &
	&\Ad_y(\Gamma) = 0{}{,}
\end{align*}
so that $F_1 + \Ad_y = G_1$. Similarly, $F_0 + \Ad_{e_2}=G_0$,
so that $\HH^1(A)$ is infinite dimensional with basis the classes in $\{F_0,G_0,E_s:n\in\NN_0\}$. Moreover, for each $s,t\in \NN$,
\[ [E_s,E_t] = (s-t)E_{s+t},\] and that $[F_0,-]$ and $[G_0,-]$
are identically zero. This determines $\HH^1(A)$
as a Lie algebra: it consists of abelian algebra $\kk^2$ acting trivially on the Witt algebra.

The following derivations form a basis of the $1$-cycles
in $\Der(B)$, where unspecified values are zero,
$s\in \NN_0$, and we agree that $y^{-1}=y^{-2}=0$:
\begin{align*}
	&\Phi_s(\alpha) =   xy^s,\quad &
	&\Phi_s(\beta) = y^s z,	\quad&
	&\Phi_s(\Lambda) = \alpha y^{s-1}z,\quad&
	&\Phi_s(\Gamma) =-\alpha y^{s-2}z, \\
	&\Phi_s'(\alpha) =   0,\quad &
	&\Phi_s'(\beta) = y^{s+2}z,	\quad&
	&\Phi_s'(\Lambda) = -\alpha y^{s+1}z,\quad&
	&\Phi_s'(\Gamma) =\alpha y^sz, \\
	&\Pi_s(\alpha) =   0,\quad &
	&\Pi_s(\beta) = y^{s+2},	\quad&
	&\Pi_s(\Lambda) = \alpha y^{s+1},\quad&
	&\Pi_s(\Gamma) =\alpha y^s, \\
	&\Pi_s'(\alpha) =   xy^sz,\quad &
	&\Pi_s'(\beta) = 0,	\quad&
	&\Pi_s'(\Lambda) =0,\quad&
	&\Pi_s'(\Gamma) =0,\\
	&\Psi_s(\alpha) = 0, \quad&
	&\Psi_s(\beta) = y^{s+2}z, \quad&
	&\Psi_s(\Lambda) = -\alpha y^{s+1} z,\quad &
	&\Psi_s(\Gamma) = xy^s\beta, \\
	&\Theta_s(\alpha) = 0, \quad&
	&\Theta_s(\beta)= 0, \quad&
	&\Theta_s(\Lambda) = \alpha y^s z - xy^s\beta,\quad &
	&\Theta_s(\Gamma) = 0, \quad\\
	&\Xi_s(\alpha) = 0, \quad&
	&\Xi_s(\beta)= 0, \quad&
	&\Xi_s(\Lambda) = 0, \quad&
	&\Xi_s(\Gamma) = \Theta_s(\Lambda).
\end{align*}

Let us now compute the $1$-boundaries of $\Der(B)$. We observe
that for every $s\in\NN_0$, the elements
$\Xi_s,\Phi_s',\Psi_s,\Theta_s$ have zero
projection onto $A$ under $\alpha {}{\colon} B\longrightarrow A$, so these are boundaries. It is easy to check that the following set of
derivations completes the list of $1$-boundaries, where $s\in \NN_0$:
\[\renewcommand{\arraystretch}{1.25}\begin{array}{@{}r@{\ }l@{\hspace{15pt}}r@{\ }l@{\hspace{15pt}}r@{\ }l@{\hspace{15pt}}r@{\ }l@{}}
	X_s(\alpha) &= x y^{s+2}, &
	X_s(\beta) &= 0{}{,}  &
	X_s(\Lambda) &= x y^{s+1} \beta, &
	X_s(\Gamma) &= -xy^s \beta{}{,} \\
	Y_s(\alpha) &= 2xy^{s+1}, &
	Y_s(\beta) &= 2y^{s+1}z,  &
	Y_s(\Lambda) &= xy^s\beta-\alpha y^s z,   &
	Y_s(\Gamma) &= 0{}{.}
\end{array}\]
Moreover, we have that $\alpha(\Phi_{s+2}+\Phi_s')=0$, so $\Phi_{s+2}$
is a boundary for $s\in\NN_0$. It follows that a basis of
$H^1(\Der(B))$ is given by the classes of the derivations
$\Phi_0,\Phi_1$ so that $\HH^2(A)$ is two dimensional.

Every derivation of degree $-2$ is
a cycle and vanishes on every generator except, possibly,
$\Lambda$ and $\Gamma$, so that a basis for the $2$-cycles is
given by the following family of derivations, where $s\in\NN_0$ and
$t\in\{0,1\}$:
\[\begin{array}{@{}r@{\ }l@{\qquad\qquad}r@{\ }l@{\qquad\qquad}r@{\ }l@{\qquad\qquad}r@{\ }l@{}}
\Omega_s^t(\Lambda) &=0, &  \Omega_s^t(\Gamma) &= xy^sz^t, &
\Uupsilon_s^t(\Lambda) &= xy^sz^t, & \Uupsilon_s^t(\Gamma) &=0.
\end{array}\]
It is straightforward to check that all of these are boundaries
except for $\Upsilon_0^1$ and $\Upsilon_0^0$. More precisely, the
following is a complete list of the $2$-boundaries, where $s\in\NN_0$
and $t\in\{0,1\}$:
\[\begin{array}{@{}r@{\ }l@{\qquad\quad}r@{\ }l@{\qquad\quad}r@{\ }l@{\qquad\quad}r@{\ }l@{}}
\Omega_s^t(\Lambda) &=0, & \Omega_s^t(\Gamma) &= xy^sz^t, &
\Uupsilon_{s+1}^t(\Lambda) &= xy^{s+1}z^t,
& \Uupsilon_{s+1}^t(\Gamma) &=0.
\end{array}\]
From this it follows that $\HH^3(A)$ is also
two dimensional.
Finally, we compute the Gerstenhaber algebra structure.
We already determined the bracket in $\HH^1(A)$, while the bracket in
$\HH^2(A)$ is trivial, since both generators vanish on $\Lambda$ and
$\Gamma$.
The action of $\HH^1(A)$ on $\HH^2(A)$ and $\HH^3(A)$ is as follows, where
$s,t\in\NN_0$ and $r\in\{0,1\}$.
\[\renewcommand{\arraystretch}{1.35}\begin{array}{@{}r@{\ }l@{\qquad\qquad}r@{\ }l@{}}
[E_{s+2},\Phi_t] &= 3\Xi_{s+t+1}-2\Theta_{s+t},&
[F_{s+2},\Phi_t] &= \Theta_{t+s+1}-\Xi_{t+s},\\
{}[G_{s+2},\Phi_t] &= \Theta_{t+s+2}-\Xi_{t+s+2}, &
 [F_1,\Phi_t] &= [G_1,\Phi_t]  = \Theta_t,  \\
{} [E_s,\Uupsilon_t^r] &= (t-3\delta{\!}_{s,0}) \Uupsilon_{s+t}^r{}{,} &
 [E_s,\Omega_t^r] &= (t+2\delta{\!}_{s,0}) \Omega_{s+t}^r, \\
{} [F_0,-] & \multicolumn{3}{@{}l}{=[G_0,-] =  2{}{,} 
 \text{ on } \langle \Omega_s^t,{\Omega'}^t_s:s\in\NN_0 \rangle,} \\
{} [F_0,-] & \multicolumn{3}{@{}l}{=[G_0,-] =  [E_0,-] = 0{}{,} 
\text{ on } \langle \Uupsilon_s^t,{\Uupsilon'}^t_s,\Phi_s:s\in\NN_0\rangle,}\\
{} [E_1,\Phi_t] &= 3\Xi_t.
\end{array}\]

 To compute cyclic and Hochschild homology of $A$,
we begin by noting that for $i\in \NN$, we have that
$[B,B]_i = B_i$, and that
$(B/[B,B])_0$ is generated
by the classes of $y^j$ for $j\in\NN_0$.
 This means that $\HC_*(A)$
is concentrated in degree zero where it coincides with
$A/[A,A] = \kk[\bar{y}]$. The long exact sequence shows that $\HH_*(A)$ is trivial in degrees larger than
$1$, and that $d{}{\colon} \HH_0(A)\to \HH_1(A)$ is
an isomorphism.

\textit{The case of dg algebras that are not `aspherical'.}
In the
above, we focused our attention on cofibrant algebras $B$ for which
$H_*(B)$ is concentrated in degree zero. It is important to point
out the work done here works equally well if $B$ is only assumed
to be cofibrant and non-negatively graded. Indeed, the only
crucial property we used is that $B$ is free as an associative
algebra, so that the complexes of fields and forms compute what
they are supposed to.

\end{document}